\documentclass{amsart}
\usepackage{hyperref}
\usepackage{txfonts}
\usepackage{amsfonts}
\usepackage{hyperref}
\usepackage{amsfonts}
\usepackage{mathrsfs}

\newtheorem{theorem}{Theorem}[section]
\newtheorem{lemma}[theorem]{Lemma}
\newtheorem{proposition}[theorem]{Proposition}

\theoremstyle{definition}
\newtheorem{definition}[theorem]{Definition}
\newtheorem{example}[theorem]{Example}

\theoremstyle{remark}

\numberwithin{equation}{section}

\begin{document}

\title{Evolution equations of curvature tensors along the hyperbolic geometric flow }

\author{Wei-Jun Lu}
\address{Center of Mathematical Sciences, Zhejiang University,
 Hangzhou,Zhejiang, 310027, P. R. China. }
 \email{weijunlu2008@126.com}


\date{}



\begin{abstract}
{\footnotesize  We consider the hyperbolic geometric flow
$\frac{\partial^2}{\partial t^2}g(t)=-2Ric_{g(t)}$ introduced by
Kong and Liu \cite{KL}. When the Riemannian metric evolve, then so
does its curvature.  Using the techniques and ideas of S.Brendle
\cite{Br,BS}, we derive evolution equations for the Levi-Civita
connection and the curvature tensors along the hyperbolic geometric
flow. The method and results are computed and written in global
tensor form, different from the local normal coordinate method in
\cite{DKL1}. In addition, we further show that any solution to the
hyperbolic geometric flow that develops a singularity in finite time
has unbounded Ricci curvature.}
\end{abstract}

\maketitle

\section{Introduction}
Geometric flows are important in many fields of mathematics and
physics. A geometric flow is an evolution of a geometric structure
under a differential equation related to a functional on a manifold,
usually associated with some curvature. The most popular geometric
flows in mathematics are the heat flow (\cite{Jo},\cite{MR}, etc.),
the Ricci flow (\cite{BLN}, \cite{Br}, \cite{To}, etc.), the mean
curvature flow (\cite{Zh}) and the Yamabe flow (\cite{Br1, Br2, Ch1,
SS, Ye}).

The subject of Hamilton's Ricci flow
\cite{Ha}:$\frac{\partial}{\partial t}g(t)=-2Ric_{g(t)}$ lies in the
more general field of geometric flows, which in turn lies in the
even more general field of geometric analysis. In Ricci flow we see
the unity of geometry and analysis. As a fully nonlinear system of
parabolic partial differential equations of second order
(\cite{De}), the Ricci flow in many respects appears to be very
natural equation.

Similarly, since the hyperbolic equation or system is one of the
most natural models in the nature, we feel the hyperbolic geometric
flow (for short, HGF):$\frac{\partial^2}{\partial
t^2}g(t)=-2Ric_{g(t)}$ introduced by Kong and Liu \cite{KL} in the
year 2007, is also a very natural tool. Note that the elliptic and
parabolic partial differential equations have been successfully
applied to differential geometry and physics. A natural and
important question is if we can apply the well-developed theory of
hyperbolic differential equations to solve problems in differential
geometry and theoretical physics. HGF is helpful to understand the
wave character of the metrics, wave phenomenon of the curvatures,
the evolution of manifolds and their structures (see \cite{KL, DKL1,
Ko}). For more discussions for related hyperbolic flows and their
applications to geometry and Einstein equations, we refer to
\cite{Li}.

In geometry, singularity is more difficult to define, especially
when the structure is governed by the hyperbolic system and the
topology of the space is allowed to change. A good example is the
singularity developed in general relativity. The major problem here
is that one has very little understanding of the global behavior of
nonlinear hyperbolic systems when the spatial dimension is greater
than one. It is almost for sure that a break through will be
accomplished in geometry if one knows this type of equation better
(See \cite{SY}). HGF equation is quite difficult to solve in all
generality unlike the nonlinear elliptic problems with a well
developed regularity theory. Although the short time existence of
solutions is guaranteed by hyperbolic nature of the equations, their
(long time) convergence to canonical geometric structures is
analyzed under various conditions. So far, providing results of this
sort have been developed. Among them, one may find the related works
in \cite{DKL2, He, KLX, KLW}.

The goal in this paper is to work out the evolution equations for
Riemannian, Ricci and scalar curvature under the hyperbolic
geometric flow. More precisely, we will concentrate in obtaining the
global forms of the evolutions under HGF. As for the local forms of
the evolutions, Dai-Kong-Liu \cite{DKL1} have obtained the following
results.

  To state our results, it will be convenient to assume that $(M, g(t)), t \in (0, T)$, is
a family of complete Riemannian manifolds evolving under hyperbolic
geometric flow. In addition, `Ric' and `scal' denote the Ricci and
scalar curvature of $(M, g(t))$, respectively.

\begin{theorem}  Under the hyperbolic geometric flow
$\frac{\partial^2}{\partial t^2}g(t)_{ij}=-2Ric_{ij}$, the curvature
tensors satisfy the evolution equations
  \begin{equation} \label{9-1}
   \begin{split}
   \frac{\partial^2}{\partial t^2} R_{ijkl}
   & =\Delta R_{ijkl}+2 (B_{ijkl}-B_{ijlk}- B_{iljk} + B_{ikjl})\\
   &-g^{pq}(R_{pjkl}Ric_{qi} + R_{ipkl}Ric_{qj} + R_{ijpl}Ric_{qk} + R_{ijkp}Ric_{ql})\\
   & +2g_{pq} \Big(\frac{\partial \Gamma^p_{il}}{\partial t}\frac{\partial \Gamma^q_{jk}}{\partial t}
       -\frac{\partial \Gamma^p_{jl}}{\partial t}\frac{\partial \Gamma^q_{ik}}{\partial t}\Big),
   \end{split}
   \end{equation}

\begin{equation} \label{9-2}
   \begin{split}
   \frac{\partial^2}{\partial t^2} Ric_{ik}
   & =\Delta Ric_{ik}+2g^{pr}g^{qs}R_{piqk}Ric_{rs}- 2g^{pq}Ric_{pi}Ric_{qk}\\
   & +2g^{jl} g_{pq}\Big(\frac{\partial \Gamma^p_{il}}{\partial t}\frac{\partial \Gamma^q_{jk}}{\partial t}
       -\frac{\partial \Gamma^p_{jl}}{\partial t}\frac{\partial \Gamma^q_{ik}}{\partial t}\Big)\\
   &-2g^{jp}g^{lq}\frac{\partial g_{pq}} {\partial t} \frac{\partial }{\partial t}R_{ijkl}
    + 2g^{jp}g^{rq}g^{sl}\frac{\partial g_{pq}} {\partial t}\frac{\partial g_{rs}} {\partial t}R_{ijkl},
   \end{split}
   \end{equation}

   \begin{equation} \label{9-3}
   \begin{split}
   \frac{\partial^2}{\partial t^2}\ Scal &=\Delta\ Scal+2|Ric|^2 \\
   &+2g^{ik}g^{jl} g_{pq}\Big(\frac{\partial \Gamma^p_{il}}{\partial t}\frac{\partial \Gamma^q_{jk}}{\partial t}
        -\frac{\partial \Gamma^p_{jl}}{\partial t}\frac{\partial \Gamma^q_{ik}}{\partial t}\Big)\\
   &-2g^{ik}g^{jp}g^{lq}\frac{\partial g_{pq}} {\partial t} \frac{\partial }{\partial t}R_{ijkl}\\
  & -2g^{ip}g^{kq}\frac{\partial g_{pq}} {\partial t} \frac{\partial Ric_{ik}}{\partial t}
    + 4Ric_{ik}g^{ip}g^{rq} g^{sk}\frac{\partial g_{pq}} {\partial t} \frac{\partial g_{rs}}{\partial t},
    \end{split}
   \end{equation}
where $\{x^i\}$ is a local normal coordinates around a fixed point
$p\in M$, $B_{ijkl} = g^{pr}g^{qs}R_{piqj}R_{rksl}$ and $\Delta$ is
the Laplacian with respect to the evolving metric $g(t)$.
\end{theorem}

Analogous results for the Einstein's hyperbolic geometric flow
(\ref{1-3}) and the dissipative hyperbolic geometric flow
(\ref{1-4}) occur in (\cite{He}) and (\cite{DKL2}), respectively.

  Motivated by the techniques and ideas concerning the Ricci flow in S. Brendle's paper
  \cite{Br, BS}, we obtain the following main results with global forms under the HGF.

 \begin{theorem} \label{9-4}
 Let $X,Y,Z,W$ be fixed vector fields on $M$. Then under the HGF (\ref{1-1}), we
 have
  \begin{equation} \label{9-5}
  \begin{split}
   \frac{\partial^2}{\partial t^2} R(X,Y,Z,W)& =-Ric_{g(t)}(R(X,Y)Z,W)+Ric_{g(t)}(R(X,Y)W,Z)\\
        & +(D^2_{X,Z} Ric_{g(t)})(Y,W)-(D^2_{X,W} Ric_{g(t)})(Y,Z)\\
                   & -(D^2_{Y,Z} Ric_{g(t)})(X,W)+(D^2_{Y,W} Ric_{g(t)})(X,Z)\\
                  &+2\frac{\partial g(t)}{\partial t}\big((D_X B)(Y,Z),W\big)
                    -2\frac{\partial g(t)}{\partial t}\big((D_Y B)(X,Z),W\big)\\
         &-2g(t)\big(B(X,B(Y,Z)),W \big)+2g(t)\big( B(Y,B(X,Z)),W \big),
    \end{split}
  \end{equation}
where $B(X,Y):=\frac{\partial}{\partial t}(D_X Y)$.
\end{theorem}

\begin{theorem}\label{9-6} Let $X,Y,Z,W$ be arbitrary fields on $M$. Then
 under the hyperbolic geometric flow (\ref{1-1}), the curvature tensors satisfy the
 evolution equations
\begin{equation} \label{9-7}
   \begin{split}
   \frac{\partial^2}{\partial t^2} R(X,Y,Z,W)
   & =(\Delta R)(X,Y,Z,W)+Q(R)(X,Y,Z,W)\\
   &  \quad -Ric_{g(t)}(X,R_{Z,W}Y)+Ric_{g(t)}(Y, R_{Z,W}X)\\
         &+Ric_{g(t)}(Z,R_{X,Y}W)-Ric_{g(t)}(W, R_{X,Y}Z)\\
                  &+2\frac{\partial g(t)}{\partial t}\big((D_X B)(Y,Z),W\big)
                    -2\frac{\partial g(t)}{\partial t}\big((D_Y B)(X,Z),W\big)\\
         &-2g(t)\big(B(X,B(Y,Z)),W \big)+2g(t)\big( B(Y,B(X,Z)),W \big),
   \end{split}
   \end{equation}
where $Q(R):=R^2+R^\#$ is a curvature tensor satisfying the first
Bianchi identity, and given a local orthonormal basis
$\{e_1,\ldots,e_n\}$, $R^2$ and $R^\#$ are difined by
\begin{equation} \label{9-8}
\begin{array}{ll}
  R^2(X,Y,Z,W):=\sum\limits_{p,q=1}^{n}R(X,Y,e_p,e_q)R(e_p,e_q,Z,W),\\
  R^\#(X,Y,Z,W):=2\sum\limits_{p,q=1}^{n}\Big\lvert \begin{array}{cc} R(X,e_p,Z,e_q)& R(X,e_p,W,e_q)\\
                                                 R(Y,e_p,Z,e_q)&R(Y,e_p,W,e_q)
                                                 \end{array}  \Big\rvert.\\
   \end{array}
   \end{equation}
\end{theorem}

\begin{theorem}\label{9-9}  Let $X,Y$ be arbitrary fields on $M$. Then
 under the hyperbolic geometric flow (\ref{1-1}), the Ricci curvature tensor $Ric_{g(t)}$
 and scalar curvature $scal_{g(t)}$ satisfy the
 evolution equations, respectively
 \begin{equation} \label{9-10}
   \begin{split}
   \frac{\partial^2}{\partial t^2} Ric_{g(t)}(X,Y)
   & =(\Delta Ric_{g(t)})(X,Y)+2\sum\limits_{i,j=1}^{n} R(X,e_i,Y,e_j) Ric_{g(t)}(e_i,e_j)\\
   &+2\sum\limits_{i=1}^{n}\Big(\frac{\partial g(t)}{\partial t}\big((D_X B)(e_i,Y),e_i\big)
                    -\frac{\partial g(t)}{\partial t}\big((D_{e_i} B)(X,Y),e_i\big)\Big)\\
   &-2\sum\limits_{i=1}^{n}\Big( g(t)\big(B(X,B(e_i,Y)),e_i \big)
         -g(t)\big( B(e_i,B(X,Y)),e_i \big) \Big),
   \end{split}
   \end{equation}

   \begin{equation} \label{9-11}
   \begin{split}
   \frac{\partial^2}{\partial t^2} Scal_{g(t)}
          &= \Delta Scal_{g(t)}+2|Ric_{g(t)}|^2\\
          &\quad +2\sum\limits_{i,j=1}^{n}\Big( \frac{\partial g(t)}{\partial t}\big((D_{e_j} B)(e_i,e_j),e_i)\big)
          -\frac{\partial g(t)}{\partial t}\big((D_{e_i}
          B)(e_j,e_j),e_i\big) \Big)\\
         & -2\sum\limits_{i,j=1}^{n}\Big( g(t)\big(B(e_j,B(e_i,e_j)),e_i \big)
         -g(t)\big( B(e_i,B(e_j,e_j)),e_i \big)\Big),
   \end{split}
   \end{equation}
where $\{e_1,\ldots,e_n\}$ is a local orthonormal basis of $M$, and
$\mid Ric \mid^2=\sum\limits_{i,j=1}^{n}(Ric(e_i,e_j))^2$.
\end{theorem}

The structure of the paper is as follows. In Section 2 we state the
related concepts such as HGF, Einstein HGF and present some examples
of specific solutions to the HGF. In Section 3 we give evolution
equations for the Levi-Civita connection and prove the main results
in the introduction. In Section 4 we consider Ricci curvature
blow-up at finite-times singularities. Section 5 presents some
expected problems.

\section{Hyperbolic geometric flow}

Recall that Kong and Liu \cite{KL} introduce a geometric
flow---Hyperbolic geometric flow (HGF), which is difference from the
Hamilton's Ricci flow, although these two flows share a common Ricci
term $-2Ric_{g(t)}$. The definition of HGF is as follows.

\begin{definition} \label{1-0} Let $M$ be a manifold. The hyperbolic geometric flow is the evolution
  \begin{equation} \label{1-1}
  \frac{\partial^2}{\partial t^2}g(t)=-2Ric_{g(t)}
  \end{equation}
 for a one-parameter family of Riemannian metrics $g(t),\ t\in [0,T)$ on $M$. We say that $g(t)$ is a
  solution to the hyperbolic geometric flow if it satisfies (\ref{1-1}).
\end{definition}

  Similar to the Ricci flow $\frac{\partial}{\partial t}g(t)=-2Ric_{g(t)}$,
the HGF equation (\ref{1-1}) is a unnormalized evolution equation.
In \cite{KL}, Kong and Liu also consider the normalized version of
hyperbolic geometric flow, which preserves the volume of the flow.
Considering the HGF and the normalized HGF differ only by a change
of scale in space $M$ and a change of time $t$, the {\em normalized
HGF} equation reads such form as
  \begin{equation} \label{1-2}
  \frac{\partial^2}{\partial t^2}g(t)
  =-2Ric_{g(t)}+a(t)\frac{\partial}{\partial t}g(t)+b(t)g(t),
  \end{equation}
where $a(t)$ and $b(t)$ are certain functions of $t$.

In order to further understand the relationship between the Einstein
equation and the HGF, Kong and Liu also introduce a so-called
Einstein's hyperbolic geometric flow.

\begin{definition} Let $\mathbb R \times M $ be a space-time with the
Lorentzian metric $ ds^2 = dt^2 + g_{ij}(x,t)dx^i dx^j$. Suppose the
Einstein equations in the vacuum, which correspond to the metric
$ds^2$, has the form
   \begin{equation} \label{1-3}
   \frac{\partial^2}{\partial t^2}g(t)_{ij}=- 2Ric_{ij}
   - \frac{1}{2} g^{pq}\frac{\partial g_{ij}}{\partial t} \frac{\partial g_{pq}}{\partial t}
   +g^{pq}\frac{\partial g_{ip}}{\partial t}\frac{\partial g_{jq}}{\partial t}.
\end{equation}
 The equation (\ref{1-3}) is called Einstein's hyperbolic geometric flow.
\end{definition}

Motivated by the well-developed theory of the dissipative hyperbolic
equations, Dai-Kong-Liu \cite{DKL2} introduce a new geometric
 analytical tool---{\em dissipative hyperbolic geometric flow} defined by
\begin{equation} \label{1-4}
  \begin{split}
\frac{\partial^2}{\partial t^2}g(t)_{ij} & =- 2Ric_{ij}
   +2 g^{pq}\frac{\partial g_{ip}}{\partial t} \frac{\partial g_{jq}}{\partial t}
   -\big(d + 2g^{pq}\frac{\partial g_{pq}}{\partial t} \big) \frac{\partial g_{ij}}{\partial t}\\
   &+\frac{ 1}{ n-1} \Big( \big(g^{pq}\frac{\partial g_{pq}}{\partial t}\big)^2
      + \frac{\partial g^{pq}}{\partial t}\frac{\partial g_{pq}}{\partial t}\Big)g_{ij},
  \end{split}
  \end{equation}
where $d$ is a positive constant. The reason that (\ref{1-4}) is
chosen as the equation form of dissipative hyperbolic geometric flow
is that, in the case it possesses a simpler equation satisfied by
the scalar curvature.

In order to get a feel for the HGF (\ref{1-1}), we present some
examples of specific solutions (cf. \cite{KL,DKL1,DKL2, KLW, He}).

\begin{example} (i)(Trivial example) If the initial metric is Ricci flat,
so that $Ric_{ij}=0$, then clearly the metric does not change under
(\ref{1-1}). Hence any Ricci flat metric $g(t)$ is a
stationary solution to the hyperbolic geometric flow. This happens, for example,
on the flat torus or on any K3-surface with a Calabi-Yau metric.\\
(ii)(Non-trivial example) A typical example of the Einstein metric
is
  $$ ds_0^2 = \frac{1}{1-\kappa r^2}dr^2 + r^2 d\theta^2 + r^2 \sin^2\theta d \varphi^2,$$
where $\kappa$ is a constant taking its value $-1, 0$ or $1$. We can
prove that the metric
  $$ds^2 = (-2\kappa t^2 + c_1 t + c_2)ds_0^2 $$
is a solution to the HGF (\ref{1-1})), where $c_1$ and $c_2$ are two
constants.
\end{example}

\begin{example} Consider the solution in conformal class  with the following form
   \begin{equation} \label{1-5} g_{ij}(t, x) = \rho(t)g_{ij}(0, x).
   \end{equation}
Suppose that the initial metric $g_{ij}(0, x)$ is Einstein, that is,
there exists some constant $\lambda$ such that
    $$Ric_{ij}(0, x)=\lambda g_{ij}(0, x), \forall x \in M.$$
Then  (\ref{1-5}) with $\rho(t) = -\lambda t^2 + vt + 1$ and a real
number $v$ standing for the initial velocity, is a solution to the
equation (\ref{1-1}).
\end{example}

\begin{example} For the Einstein's hyperbolic geometric flow
(\ref{1-3}), its exact solution has the following form
  \begin{equation}\notag
   ds^2 = f(t,z)dz^2
     + \frac{h(t)}{ F(t, z)} \big((dx-\mu(t, z)dy) ^2 + F^2(t, z)dy^2 \big),
   \end{equation}
where $f,h$ and $F$ are smooth functions with respect to variables.
The $x$-invariance and $y$-invariance show that the model possesses
the $z$-axial symmetry. In order to guarantee that the metric
$g_{ij}$ is Riemannian, we assume $F(t, z) > 0$ and $h(t)/F(t,z)>0.$
\end{example}

More interesting example comes from the Riemann surfaces with the
initial asymptotic flat.

\begin{example}  On a surface, the HGF equation (\ref{1-1}) can be simplified as the
following equation for the special metric
      \begin{equation} \label{11-1}
       \frac{\partial^2}{\partial t^2}g_{ij}(t, x,y) = -Scal\ g_{ij}(t, x, y).
       \end{equation}
Because the Ricci curvature is given by
    \begin{equation}\notag
    Ric_{ij}(t,x,y)=\frac{1}{2}Scal(t,x,y)\ g_{ij}(t,x,y),
   \end{equation}
 where $Scal= 2K$ is the scalar curvature function with the Gauss
 curvature $K$. Note that the metric for a surface can always be written (at least locally)
 as the following form
          \begin{equation} \notag
        g_{ij}(t,x,y)=u(t,x,y)\delta_{ij},
   \end{equation}
where $u=u(t,x,y) > 0$. Thus
     $$ Scal =-\frac{1}{u}\Delta \log{ u}.$$
This implies that (\ref{11-1}) reads
       \begin{equation} \label{11-2}
       u_{tt}- \Delta \log{u}=0.
       \end{equation}
Define $f=\log{u}$, then (\ref{11-2}) exchanges a quasilinear
hyperbolic wave equation
        \begin{equation} \label{11-3}
       f_{tt}-e^{-f} \Delta u=-f_t.
       \end{equation}
Consider the Cauchy problem for (\ref{11-3}) with the following
initial data
    \begin{equation} \label{11-4}
   t = 0 :f=\varepsilon f_0(x), f_t = \varepsilon f_1(x)
       \end{equation}
and satisfying the slow decay condition
    \begin{equation} \label{11-5}
     |f_0(x)|\leq \frac{A}{(1+|x|)^k},\, |f_1(x)|=\frac{A}{(1+|x|)^{k+1} },
       \end{equation}
where $\varepsilon >0$ is a suitably small parameter, and
$f_0(x),f_1(x)\in C^{\infty}( \mathbb{R}^2)$; $A, k$ are two
suitable positive constants, $k> 1$. Then there exist two positive
constants $\delta$ and $\varepsilon_0$ such that for any fixed
$\varepsilon\in [0,\varepsilon_0]$, the Cauchy problem
$(\ref{11-3})\sim (\ref{11-5})$ has a unique $C^{\infty}$ solution
on the interval $[0,T_{\varepsilon}]$ with $T_{\varepsilon}=\delta
\varepsilon^{-4/3}$ (for detail, see \cite{KLW}).
\end{example}

\section{Evolution of the curvature tensors and proofs of main results}

 In this section, we derive evolution equations with global forms for the Levi-Cicvita
 connection and the curvature tensors along the HGF. We employ the
 techniques and ideas in studying evolution equations along the Ricci flow by
 S.Brendle (See \cite{Br} or \cite{BS}).

 From now on, we assume that $(M, g(t)), t\in (0, T)$, is a family of complete Riemannian
 manifolds evolving under HGF.

\subsection{Evolution of the Levi-Civita connection }
  Let $X,Y$ be fixed vector fields on $M$ (that is, $X,Y$ are independent of
  $t$). We define
       $$A(X,Y):=\frac{\partial^2}{\partial t^2}(D_X Y),\, B(X,Y):=\frac{\partial}{\partial t}(D_X
       Y).$$
Observe that the difference of two connections is always a tensor,
consequently, $A,B$ are tensors.

\begin{proposition} \label{6-1}
 Let $X,Y,Z$ be fixed vector fields on $M$. Then
  \begin{equation} \label{6-2}
  \begin{split}
    g(t)(A(X,Y),Z)& =-(D_X Ric_{g(t)})(Y,Z)-(D_Y Ric_{g(t)})(X,Z)\\
         & +(D_Z  Ric_{g(t)})(X,Y)-2\frac{\partial g(t)}{\partial t}(B(X,Y),Z).
    \end{split}
  \end{equation}
\end{proposition}

\begin{proof} For $g(t)(D_X Y,Z)$, we differentiate it twice with
respect to $t$. This yields
  \begin{equation} \label{6-3}
    \frac{\partial^2}{\partial t^2}( g(t)(D_X Y,Z) )
         =\frac{\partial^2 g(t)}{\partial t^2}(D_X Y,Z)+2\frac{\partial g(t)}{\partial
                  t}(B(X,Y),Z)+g(t)(A(X,Y),Z).
     \end{equation}
Since the Levi-Civita connection satisfies
  \[ \begin{split}
   2g(t)(D_X Y, Z)&=X(g(t)(Y,Z))+Y(g(t)(Z,X))-Z(g(t)(X,Y))\\
      &\quad -g(t)(X,[Y,Z])+g(t)(Y,[Z,X])+g(t)(Z,[X,Y]),
   \end{split}
  \]
(\ref{6-3}) can be rewritten in form
  \begin{equation} \notag
  \begin{split}
    g(t)(A(X,Y),Z)& =X(\frac{1}{2} \frac{\partial^2 g(t)}{\partial t^2}(Y,Z))
                    +Y(\frac{1}{2} \frac{\partial^2 g(t)}{\partial t^2}(Z,X))\\
     &-Z(\frac{1}{2} \frac{\partial^2 g(t)}{\partial t^2}(X,Y))
             -\frac{1}{2} \frac{\partial^2 g(t)}{\partial t^2}(X,[Y,Z]) \\
     &+\frac{1}{2} \frac{\partial^2 g(t)}{\partial t^2}(Y,[Z,X])
        +\frac{1}{2} \frac{\partial^2 g(t)}{\partial t^2}(Z,[X,Y])\\
     & - \frac{\partial^2 g(t)}{\partial t^2}(D_XY,Z)-2\frac{\partial g(t)}{\partial t}(B(X,Y),Z).
    \end{split}
  \end{equation}
By definition of the HGF, we have
  \begin{equation} \label{6-4}
  \begin{split}
    g(t)(A(X,Y),Z)& =-X(Ric(Y,Z)) -Y(Ric(Z,X))+Z(Ric(X,Y))\\
     &+Ric(X,[Y,Z])-Ric(Y,[Z,X])-Ric(Z,[X,Y])\\
     & +2 Ric(D_XY,Z)-2\frac{\partial g(t)}{\partial t}(B(X,Y),Z).
    \end{split}
  \end{equation}
Note the $A$ is a tensor, we conclude that
   \begin{equation} \label{6-4}
  \begin{split}
    g(t)(A(X,Y),Z)& =-(D_X Ric(Y,Z) -(D_Y Ric)(Z,X)\\
    &+(D_Z Ric)(X,Y) -2\frac{\partial g(t)}{\partial t}(B(X,Y),Z),
    \end{split}
  \end{equation}
as claimed.
\end{proof}

\subsection{Proof of main result 1: Theorem \ref{9-4}}

Now we return to compute the evolution equation for the curvature
tensor. For convenience we need the second order covariant
derivative $D^2_{X,Y}Z$ defined by
    \begin{equation} \notag
     D^2_{X,Y}Z:=D_X D_Y Z-D_{D_X Y}Z,
    \end{equation}
from which we have
   \begin{equation} \notag
     R(X,Y)Z:=D_XD_Y Z-D_YD_X Z-D_{[X,Y]}Z=D^2_{X,Y}Z-D^2_{Y,X}Z.
 \end{equation}

\begin {proof}[Proof of Theorem \ref{9-4}] The second derivative of $R(X,Y)Z$ yields
    \begin{equation} \label{7-3}
    \begin{split}
    \frac{\partial^2}{\partial t^2}R(X, Y)Z
        & =\frac{\partial^2 }{\partial t^2}D_X (D_Y Z)
            +2\frac{\partial }{\partial t}D_X(\frac{\partial }{\partial t}D_Y Z)
              +D_X(\frac{\partial^2 }{\partial t^2}D_Y Z)\\
      & -\frac{\partial^2 }{\partial t^2}D_Y (D_X Z)
            -2\frac{\partial }{\partial t}D_Y(\frac{\partial }{\partial t}D_X Z)
              -D_Y(\frac{\partial^2 }{\partial t^2}D_X Z)-\frac{\partial^2 }{\partial t^2}D_{[X,Y]} Z\\
     & =A(X,D_Y Z)+2B(X,B(Y,Z))+D_X A(Y,Z)-A(Y,D_X Z)\\
      & -2B(Y,B(X,Z))-D_Y A(X,Z)-A(D_X Y-D_Y X,Z)\\
     & =(D_X A)(Y,Z)-(D_Y A)(X,Z)+2B(X,B(Y,Z))-2B(Y,B(X,Z)).
    \end{split}
     \end{equation}
This implies
     \begin{equation} \label{7-4}
    \begin{split}
    &\frac{\partial^2}{\partial t^2}R(X,Y,Z,W)\\
        & =\frac{\partial^2}{\partial t^2}\big(g(t)(-R(X,Y)Z,W \big)\\
    & = -\frac{\partial^2 g(t) }{\partial t^2}(R(X,Y)Z,W)-g(t)\big(\frac{\partial^2}{\partial t^2}R(X,Y)Z,W\big)\\
      & =2Ric_{g(t)}(R(X,Y)Z,W)-g(t)((D_X A)(Y,Z),W)+g(t)((D_Y A)(X,Z),W)\\
      & -2g(t)(B(X,B(Y,Z)),W)+2g(t)(B(Y,B(X,Z)),W).
    \end{split}
     \end{equation}
Applying Proposition \ref{6-1}, we obtain
    \begin{equation} \label{7-5}
    \begin{split}
    & g(t)((D_X A)(Y,Z),W)\\
    & =X\big(g(t)(A(Y,Z),W)\big)-g(t)(A(Y,Z),D_X W)\\
     & -g(t)\big(A(D_X Y,Z),W\big)-g(t)\big(A(Y,D_X Z),W\big)\\
     & =X\Big(-(D_Y Ric_{g(t)})(Z,W)-(D_Z Ric_{g(t)})(Y,W)+(D_W  Ric_{g(t)})(Y,Z)\\
     &  \quad -2\frac{\partial g(t)}{\partial t}(B(Y,Z),W)\Big)
       +(D_Y Ric_{g(t)})(Z,D_X W)+(D_Z Ric_{g(t)})(Y,D_X W)\\
     & -(D_{D_X W} Ric_{g(t)})(Y,Z)+2\frac{\partial g(t)}{\partial t}(B(Y,Z),D_X W)\\
     & +(D_{D_X Y} Ric_{g(t)})(Z,W)+(D_Z Ric_{g(t)})(D_X Y, W)\\
     & -(D_{W} Ric_{g(t)})(D_X Y,Z)+2\frac{\partial g(t)}{\partial t}(B(D_X Y,Z),W)\\
      &+(D_Y Ric_{g(t)})(D_X Z,W)+(D_{D_X Z} Ric_{g(t)})(Y,W)\\
         & -(D_{W} Ric_{g(t)})(Y,D_X Z)+2\frac{\partial g(t)}{\partial t}(B(Y,D_X Z),W)\\
    \end{split}
     \end{equation}
    \begin{equation} \notag
    \begin{split}
     &= -(D_XD_Y Ric_{g(t)}-D_{D_X Y}Ric_{g(t)})(Z,W)-(D_XD_Z Ric_{g(t)}-D_{D_X Z}Ric_{g(t)})(Y,W)\\
     &+(D_XD_W Ric_{g(t)}-D_{D_X W}Ric_{g(t)})(Y,Z)-2\frac{\partial g(t)}{\partial t}\big((D_X B)(Y,Z),W\big)\\
     &=-(D^2_{X,Y} Ric_{g(t)})(Z,W)-(D^2_{X,Z} Ric_{g(t)})(Y,W)\\
     & +(D^2_{X,W} Ric_{g(t)})(Y,Z)-2\frac{\partial g(t)}{\partial t}\big((D_X B)(Y,Z),W\big).
    \end{split}
     \end{equation}
Interchanging the roles of $X$ and $Y$ yields
  \begin{equation} \label{7-6}
    \begin{split}
    & g(t)((D_Y A)(X,Z),W)\\
   & =-(D^2_{Y,X} Ric_{g(t)})(Z,W)-(D^2_{Y,Z} Ric_{g(t)})(X,W)\\
    &  +(D^2_{Y,W} Ric_{g(t)})(X,Z)-2\frac{\partial g(t)}{\partial t}\big((D_Y B)(X,Z),W\big).
    \end{split}
     \end{equation}
Moreover, we have
   \begin{equation} \label{7-7}
    \begin{split}
     &(D^2_{X,Y} Ric_{g(t)})(Z,W)-(D^2_{Y,X} Ric_{g(t)})(Z,W)\\
     & =\big((D^2_{X,Y}-D^2_{Y,X})Ric_{g(t)}\big)(Z,W )\\
     & =\big( R(X,Y) Ric_{g(t)}\big)(Z,W )\\
     & =Ric_{g(t)}(R(X,Y)Z,W)+Ric_{g(t)}(R(X,Y)W,Z).
    \end{split}
     \end{equation}
Substituting (\ref{7-4}) with (\ref{7-5}), (\ref{7-6}) and
(\ref{7-7}), we get
   \begin{equation} \label{7-7'}
  \begin{split}
      &\frac{\partial^2}{\partial t^2}R(X,Y,Z,W)\\
      &=-Ric_{g(t)}(R(X,Y)Z,W)+Ric_{g(t)}(R(X,Y)W,Z)\\
     &+(D^2_{X,Z} Ric_{g(t)})(Y,W)-(D^2_{X,W} Ric_{g(t)})(Y,Z)\\
    & -(D^2_{Y,Z} Ric_{g(t)})(X,W)+(D^2_{Y,W} Ric_{g(t)})(X,Z)\\
     &+2\frac{\partial g(t)}{\partial t}\big((D_X B)(Y,Z),W\big)
     -2\frac{\partial g(t)}{\partial t}\big((D_Y B)(X,Z),W\big)\\
   & -2g(t)(B(X,B(Y,Z)),W)+2g(t)(B(Y,B(X,Z)),W),
    \end{split}
  \end{equation}
as claimed.
\end{proof}

\subsection{Proof of main result 2: Theorem \ref{9-6}}

We will show that the right-hand side in the equation (\ref{7-7'})
for the curvature tensor equals the Laplacian of the curvature
tensor, up to lower order terms. To this end, we first give the
following lemma (See \cite{Br}) which is independent of any
evolution.

\begin{lemma} \label{7-8} Let $X,Y,Z,W$ be arbitrary fields on $M$. Then
   \begin{equation}\label{7-9}
   \begin{split}
    &(D^2_{X,Z} Ric_{g(t)})(Y,W)-(D^2_{X,W}Ric_{g(t)})(Y,Z)\\
    &-(D^2_{Y,Z}Ric_{g(t)})(X,W)+(D^2_{Y,W}Ric_{g(t)})(X,Z) \\
   & =(\Delta R)(X,Y,Z,W)+Q(R)(X,Y,Z,W)-Ric_{g(t)}(X,R_{Z,W}Y)+Ric_{g(t)}(Y, R_{Z,W}X),
   \end{split}
   \end{equation}
where $Q(R):=R^2+R^\#$ is a curvature tensor satisfying the first
Bianchi identity, and given a local orthonormal basis
$\{e_1,\ldots,e_n\}$, $R^2$ and $R^\#$ are difined by (\ref{9-8}) in
Theorem \ref{9-6}.
\end{lemma}

\begin{proof}
For the orthonormal basis $\{e_1,\ldots,e_n\}$, it is easy to show
that
  \begin{equation} \label{7-11}
  \begin{split}(D^2_{X,Z}
  Ric_{g(t)})(Y,W)=\sum\limits_{k=1}^{n}(D^2_{X,Z}R)(e_k,Y,e_k,W),\\
   (D^2_{X,W}Ric_{g(t)})(Y,Z)=\sum\limits_{k=1}^{n}(D^2_{X,W}R)(e_k,Y,e_k,Z).
   \end{split}
   \end{equation}
Using the second Bianchi identity, by a direct computation we obtain
 \begin{equation} \label{7-10}
  (D^2_{X,Z}R)(e_k,Y,e_k,W)-(D^2_{X,W}R)(e_k,Y,e_k,Z)=(D^2_{X,e_k}R)(e_k,Y,Z,W).
  \end{equation}
Thus we have
    \begin{equation}\label{7-10-1}
    (D^2_{X,Z}
    Ric_{g(t)})(Y,W)-(D^2_{X,W}Ric_{g(t)})(Y,Z)=\sum\limits_{k=1}^{n}(D^2_{X,e_k}R)(e_k,Y,Z,W),
    \end{equation}
and
     \begin{equation} \label{7-10-2}
    (D^2_{Y,Z} Ric_{g(t)})(X,W)-(D^2_{Y,W}Ric_{g(t)})(X,Z)=\sum\limits_{k=1}^{n}(D^2_{Y,e_k}R)(e_k,X,Z,W).
    \end{equation}
Therefore, (\ref{7-10-1}) and (\ref{7-10-2}) yield
   \begin{equation}\label{7-12}
   \begin{split}
    &I:=(D^2_{X,Z} Ric_{g(t)})(Y,W)-(D^2_{X,W}Ric_{g(t)})(Y,Z)\\
    &-(D^2_{Y,Z}Ric_{g(t)})(X,W)+(D^2_{Y,W}Ric_{g(t)})(X,Z) \\
   & =\sum\limits_{k=1}^{n}\Big(D^2_{X,e_k}R)(e_k,Y,Z,W)-(D^2_{Y,e_k}R)(e_k,X,Z,W)\Big).
   \end{split}
   \end{equation}
Now we consider to put $Q(R)$ into (\ref{7-12}). Note that
    \begin{equation}\label{7-13}
   \begin{split}
       &\sum\limits_{k=1}^{n}(D^2_{X,e_k} R-D^2_{e_k,X}R)(e_k,Y,Z,W)\\
    & =\sum\limits_{k,l=1}^{n}\big(R(X,e_k,e_k,e_l)R(e_l,Y,Z,W)+R(X,e_k,Y,e_l)R(e_k,e_l,Z,W)\\
      & \quad     +R(X,e_k,Z,e_l)R(e_k,Y,e_l,W)+R(X,e_k,W,e_l)R(e_k,Y,Z,e_l)\big)
   \end{split}
   \end{equation}
and
\begin{equation}\label{7-14}
   \begin{split}
       &\sum\limits_{k=1}^{n}(D^2_{Y,e_k} R-D^2_{e_k,Y}R)(e_k,X,Z,W)\\
    & =\sum\limits_{k,l=1}^{n}\big(R(Y,e_k,e_k,e_l)R(e_l,X,Z,W)+R(Y,e_k,X,e_l)R(e_k,e_l,Z,W)\\
    &\quad
    +R(Y,e_k,Z,e_l)R(e_k,X,e_l,W)+R(Y,e_k,W,e_l)R(e_k,X,Z,e_l)\big),
   \end{split}
   \end{equation}
we have
    \begin{equation}\label{7-15}
   \begin{split}
   &\sum\limits_{k=1}^{n}\Big((D^2_{X,e_k} R-D^2_{e_k,X}R)(e_k,Y,Z,W)-(D^2_{Y,e_k} R-D^2_{e_k,Y}R)(e_k,X,Z,W)\Big)\\
    & =\sum\limits_{l=1}\big(Ric(X,e_l)R(e_l,Y,Z,W)-Ric(Y,e_l)R(e_l,X,Z,W)\big)\\
     &+\sum\limits_{k,l=1}^{n}\big(R(X,e_k,Y,e_l)-R(Y,e_k,X,e_l)\big)R(e_k,e_l,Z,W)\\
       & +2\sum\limits_{k,l=1}^{n}\big(R(X,e_k,Z,e_l)R(e_k,Y,e_l,W)-R(X,e_k,W,e_l)R(Y,e_k,Z,e_l)\big)\\
   \end{split}
   \end{equation}
By definitions of $R^2$ and $R^\#$, together with the first Bianchi
identity
    $$R(X,e_k,Y,e_l)-R(Y,e_k,X,e_l)=R(X,Y,e_k,e_l),$$
(\ref{7-15}) can be reduced as
     \begin{equation}\label{7-16}
   \begin{split}
   &\sum\limits_{k=1}^{n}\Big((D^2_{X,e_k} R-D^2_{e_k,X}R)(e_k,Y,Z,W)-(D^2_{Y,e_k} R-D^2_{e_k,Y}R)(e_k,X,Z,W)\Big)\\
    & =-Ric(X,R_{Z,W}Y)+Ric(Y,R_{Z,W}X)+(R^2+R^\#)(X,Y,Z,W).\\
    \end{split}
   \end{equation}
Hence, from (\ref{7-12}) and (\ref{7-16}) we have
  \begin{equation}\label{7-17}
   \begin{split}
   &I=Q(R)(X,Y,Z,W)+ \sum\limits_{k=1}^{n}\Big(D^2_{e_k,X} R)(e_k,Y,Z,W)-(D^2_{e_k,Y}R)(e_k,X,Z,W)\Big)\\
   &-Ric(X,R_{Z,W}Y)+Ric(Y,R_{Z,W}X).
    \end{split}
   \end{equation}
Next we consider to put $\Delta R$ into (\ref{7-16}). Similar to
(\ref{7-10}), we have
    $$ \sum\limits_{k=1}^{n}(D^2_{e_k,e_k} R)(X,Y,Z,W)
    =\sum\limits_{k=1}^{n}\big((D^2_{e_k,X} R)(e_k,Y,Z,W)-(D^2_{e_k,Y}R)(e_k,X,Z,W)\big).$$
Moreover, we know
    $$ \Delta R=\sum\limits_{k=1}^{n} D^2_{e_k,e_k} R.$$
Putting these facts together, (\ref{7-17}) arrives at the desired
equation.
\end{proof}

\begin{proof}[Proof of Theorem \ref{9-6}] By Theorem
\ref{9-4} and Lemma \ref {7-8}, we obtain the following
wave-character equation for curvature tensor
\begin{equation}
\label{7-18}
   \begin{split}
   \frac{\partial^2}{\partial t^2} R(X,Y,Z,W)
   & =(\Delta R)(X,Y,Z,W)+Q(R)(X,Y,Z,W)\\
   &  \quad -Ric_{g(t)}(X,R_{Z,W}Y)+Ric_{g(t)}(Y, R_{Z,W}X)\\
         &+Ric_{g(t)}(Z,R_{X,Y}W)-Ric_{g(t)}(W, R_{X,Y}Z)\\
                  &+2\frac{\partial g(t)}{\partial t}\big((D_X B)(Y,Z),W\big)
                    -2\frac{\partial g(t)}{\partial t}\big((D_Y B)(X,Z),W\big)\\
                  &-2g(t)\big(B(X,B(Y,Z)),W \big)+2g(t)\big( B(Y,B(X,Z)),W \big).
   \end{split}
   \end{equation}
\end{proof}

\subsection{Proof of main result 3: Theorem \ref{9-9}}

Notice that the hyperbolic geometric flow is an evolution equation
on the metric $g_{ij}(t)$. The evolution for the metric is not only
implies a nonlinear wave equation for the Riemannian curvature
tensor, but also for the Ricci curvature tensor and the scalar
curvature. This result (i.e.Theorems \ref{9-9}) is stated in the
introduction. Now we give its proof.

\begin{proof}[Proof of Theorem \ref{9-9}]
Similar to (\ref{7-11}), we have
\begin{equation} \label{8-3}
   \begin{split}
 (\Delta Ric_{g(t)})(X,Y)&=(\sum\limits_{j=1}^{n}(D^2_{e_j,e_j}Ric))(X,Y)\\
                          & =\sum\limits_{i,j=1}^{n}(D^2_{e_j,e_j}R)(X,e_i,Y,e_i)\\
                          & =\sum\limits_{i=1}^{n}(\Delta  R)(X,e_i,Y,e_i).
   \end{split}
   \end{equation}
By definition of $Q(R)$, we have
 \begin{equation} \label{8-4}
   \begin{split}
   \sum\limits_{i=1}^{n} Q(R)(X,e_i,Y,e_i)
   &=\sum\limits_{i,j,k=1}\Big( R(X,e_i,e_j,e_k)R(Y,e_i,e_j,e_k)\\
                          & +2R(X,e_j,Y,e_k)R(e_i,e_j,e_i,e_k)\\
                          &-2R(e_i,e_j,Y,e_k)R(X,e_j,e_i,e_k)\Big).
  \end{split}
   \end{equation}
Using the first Bianchi identity, we obtain
   \begin{equation} \notag
   \begin{split}
   &-2\sum\limits_{i,j,k=1}R(e_i,e_j,Y,e_k)R(X,e_j,e_i,e_k)\\
    &=-\sum\limits_{i,j,k=1}R(X,e_j,e_i,e_k)\big(R(Y,e_k,e_i,e_j,)-R(Y,e_i,e_k,e_j)\big)\\
    &=-\sum\limits_{i,j,k=1}R(X,e_j,e_i,e_k)R(Y,e_j,e_i,e_k)\\
     &=-\sum\limits_{i,j,k=1}R(X,e_i,e_j,e_k)R(Y,e_i,e_j,e_k).
  \end{split}
   \end{equation}
Hence (\ref{8-4}) can be reduced as
    \begin{equation} \label{8-5}
    \begin{split}
   \sum\limits_{i=1}^{n} Q(R)(X,e_i,Y,e_i)
   &=2 \sum\limits_{i,j,k=1}R(X,e_j,Y,e_k)R(e_i,e_j,e_i,e_k)\\
   &=2 \sum\limits_{j,k=1}R(X,e_j,Y,e_k)Ric_{g(t)}(e_j,e_k).
    \end{split}
    \end{equation}
Moreover, we have
  \begin{equation} \label{8-6}
   \begin{split}
  & \sum\limits_{i=1}^{n}[ -Ric_{g(t)}(X,R_{Y,e_i}e_i)+Ric_{g(t)}(e_i, R_{Y,e_i}X)\\
         &+Ric_{g(t)}(Y,R_{X,e_i}e_i)-Ric_{g(t)}(e_i, R_{X,e_i}Y)]\\
  &=\sum\limits_{i,j=1}^{n}[-R(Y,e_i,e_i,e_j)Ric_{g(t)}(X,e_j)+R(Y,e_i,X,e_j)Ric_{g(t)}(e_i,e_j)\\
  &  +R(X,e_i,e_i,e_j)Ric_{g(t)}(Y,e_j)-R(X,e_i,Y,e_j)Ric_{g(t)}(e_i,e_j)]\\
   &=   \sum\limits_{i,j=1}^{n}[Ric_{g(t)}(Y,e_j)Ric_{g(t)}(X,e_j)-Ric_{g(t)}(X,e_j)Ric_{g(t)}(Y,e_j)]\\
   &=0.
   \end{split}
   \end{equation}
 Using Theorem \ref{9-6}, or by (\ref{7-18}), together with (\ref{8-3}), (\ref{8-5}) and
(\ref{8-6}), we get
  \begin{equation} \label{8-7}
   \begin{split}
  &\frac{\partial^2}{\partial t^2} Ric_{g(t)}(X,Y)
          = \sum\limits_{i=1}^{n}\frac{\partial^2}{\partial t^2}R(X,e_i,Y,e_i)\\
  &=(\Delta Ric_{g(t)})(X,Y)+2\sum\limits_{j,k=1}R(X,e_j,Y,e_k)Ric_{g(t)}(e_j,e_k)\\
  &+2\sum\limits_{i=1}^{n}\big[\frac{\partial g(t)}{\partial t}\big((D_X B)(e_i,Y),e_i)\big)
                    -\frac{\partial g(t)}{\partial t}\big((D_{e_i} B)(X,Y),e_i\big)\big]\\
   &-2\sum\limits_{i=1}^{n}[ g(t)\big(B(X,B(e_i,Y)),e_i \big)
         -g(t)\big( B(e_i,B(X,Y)),e_i \big)],
   \end{split}
   \end{equation}
as the first claimed equality (\ref{9-10}).

As for the second assertion (\ref{9-11}), note that
  \begin{equation} \notag
   \begin{split}
   &\frac{\partial^2}{\partial t^2} Scal_{g(t)}
          = \sum\limits_{i=1}^{n}\frac{\partial^2}{\partial t^2}Ric_{g(t)}(e_i,e_i),\\
  &\sum\limits_{i=1}^{n}(\Delta Ric_{g(t)})(e_i,e_i)=\Delta scal_{g(t)},\\
  &\sum\limits_{i,j,k=1}R(e_i,e_j,e_i,e_k)Ric_{g(t)}(e_j,e_k)
   =\sum\limits_{j,k=1}Ric_{g(t)}(e_j,e_k)Ric_{g(t)}(e_j,e_k)=|Ric_{g(t)}|^2,\\
   \end{split}
   \end{equation}
it follows from (\ref{8-7}) that
   \begin{equation} \notag
   \begin{split}
   \frac{\partial^2}{\partial t^2} Scal_{g(t)}
          &= \Delta Scal_{g(t)}+2|Ric_{g(t)}|^2\\
          &\quad +2\sum\limits_{i,j=1}^{n}[\frac{\partial g(t)}{\partial t}\big((D_{e_j} B)(e_i,e_j),e_i)\big)
          -\frac{\partial g(t)}{\partial t}\big((D_{e_i}
          B)(e_j,e_j),e_i\big)]\\
         & -2\sum\limits_{i,j=1}^{n}[ g(t)\big(B(e_j,B(e_i,e_j)),e_i \big)
         -g(t)\big( B(e_i,B(e_j,e_j)),e_i \big)],
   \end{split}
   \end{equation}
as claimed.
\end{proof}

\section{Curvature blow-up at finite-time singularities }
In this section, we consider a maximal solution to the hyperbolic
geometric flow which is defined on a finite interval $[0,T)$.
Similar to the result in Hamilton's paper (\cite{Ha}, Theorem 14.1),
we show that such a solution must have unbounded curvature. The
proof is due to C. Udriste's approach (See \cite{Ud}, Theorem 6.1).

 \begin{theorem} \label{4.1} Let $M$ be a compact manifold, and let $g(t)$, $t\in
 [0,T)$, be a maximal solution to the HGF (\ref{1-1}) on $M$. Moreover, suppose
 that $T< \infty$. Then
     \begin{equation} \notag
    \limsup\limits_{t\to T} (\sup\limits_{M} |Ric_{g(t)}|)=\infty.
   \end{equation}
 \end{theorem}

\begin{proof} We argue by contradiction. Assuming the Ricci tensor of $g(t)$
is uniformly bounded for all $t\in [0,T)$, that is, there is a
positive constant m such that $|Ric_{g(t)}|\leq m,\, t\in [0,T)$. So
the solution $g(t)$ can be extended to a larger time interval
$[0,T+\varepsilon)$, where $\varepsilon$ is a arbitrary small
number. Indeed, by (\ref{1-1}), we have the relations
   \begin{equation} \notag
   \begin{split}
   & \frac{\partial g}{\partial t}(t)=\frac{\partial g}{\partial t}(0)-2 \int_0^t Ric(s,x) ds, \\
   & g(t)=g(0)+t\frac{\partial g}{\partial t}(0)-2 \int_0^t\int_0^u
   Ric(s,x) dsdu, \;t\in[0,T),
   \end{split}
   \end{equation}
which imply
   \begin{equation} \notag
   g(t_1)-g(t_2)=(t_1-t_2)\frac{\partial g}{\partial t}(0)-2 \int_{t_1}^{t_2}\int_0^u
   Ric(s,x) dsdu, \;t_1,t_2\in[0,T).
   \end{equation}
Hence we obtain
   \begin{equation} \notag
   |g(t_1)-g(t_2)|=\big(\mid \frac{\partial g}{\partial t}(0)\mid+2mT \big)|t_1-t_2|.
   \end{equation}
The Cauchy Criterion shows that $\lim_{t\to T} g(t)$ exists, while
$\lim_{t\to T }\frac{\partial g}{\partial t}(t)$ and $\lim_{t\to T
}\frac{\partial^2 g}{\partial t^2}(t)$ exist since $\lim_{t\to T
}Ric(t,x)$ exists (See Definition \ref{1-0}). Consequently, $g(t)$,
$t\in [0,T]$ is the solution to a HGF.

   In this case, $g(t)$ may be extended from being a smooth solution
   on $[0,T)$ to a smooth solution on $[0,T]$. Then we take $g(T)$,
$\frac{\partial g}{\partial t}(T)$ to be an initial metric in a
short-time existence theorem in order to extend the solution to a
HGF for $t\in [0,T+\varepsilon)$. This contradicts the assumption
that $[0,T)$ is a maximal time interval. Therefore, we complete the
proof.
\end{proof}

\section{Expected problems }
  To conclude this paper, we present some expected problems. Now we have the global forms of
evolution equations (\ref{9-5}), (\ref{9-7}),  (\ref{9-10}) and
(\ref{9-11}) along the HGF (\ref{1-1}), so in the future we are
interested in the following
problems:\\\\
 1. For fixed $(p,t)\in M\times [0,T]$, we denote by $K_{\max(p,t)}$ / $K_{\min}(p,t)$ the maximum /minimum
 sectional curvature of $g(t)$ at the point $p$. Moreover, for
 abbreviation, we define
      $$K_{\max}(t)=\sup\limits_{p\in M} K_{\max}(p,t),\; K_{\min}(t)=\inf\limits_{p\in M} K_{\min}(p,t).$$
 Let $\{t_k\}$ be a sequence of times such that $\lim\limits_{k\to \infty} t_k =T$ and
 $K_{\max}(t_k)\geq \frac{1}{2}\sup\limits_{t\in[0,t_k]}K_{\max}(t)$ for all $k$. Then
 by Theorem \ref{4.1}, does the following relation hold
   \begin{equation} \notag
   \begin{split}
   &  \limsup\limits_{k\to \infty} \frac{K_{\min}(t_k)}{  K_{\max}(t_k)  }<1,\;\; or\\
   & \lim\limits_{t\to T} \frac{K_{\min}(t)}{ K_{\max}(t) }=1?
   \end{split}
   \end{equation}
\\\\
2. ({\em Preserved curvature conditions by the HGF}) We know
  that if we want to study the global properties of HGF, then it is
  important to find curvature conditions that are preserved under
  the evolution. How to develop such techniques? For instance, suppose $M$ is
  compact manifold, and let $g(t),\ t\in [0,T)$, be a solution to
  HGF on $M$, and consider a appropriate ODE$(\ast)$
       $\frac{d^2}{dt^2}R(t)=Q(R(t))+(certain\ term)$.
Can we claim that nonnegative isometric curvature (
  \cite{Br}) is preserved by the ODE$(\ast)$?\\\\

{\bf Acknowledgments.} I am very grateful to Professor Kefeng Liu ,
as my advisor, who has encouraged me to learn the newest
developments from leading geometric experts and has been offering me
tremendous information and constructive suggestions. I would like to
thank Professor Hong-Wei Xu for helpful discussions on the course
{\em Ricci Flow and the Sphere Theorem} by S.Brendle \cite{Br}.

\end{document}